\documentclass[11pt]{amsart}
\usepackage{lineno}
\usepackage[mathbf,mathcal]{euler}
\usepackage{eucal}
\usepackage{palatino}
\usepackage{dsfont}
\usepackage{amsmath}
\usepackage{amssymb}
\usepackage{amsthm}
\usepackage{color}
\usepackage{graphics}
\usepackage{framed}
\usepackage{fullpage}
\usepackage{psfrag,graphicx,epsf}
\usepackage{hyperref}
\usepackage{xifthen}

\def\pedantic{0}
\def\pedant#1{\ifthenelse{\pedantic=1}{\textcolor{red}{ #1}}{}}
\def\achtung#1{\ifthenelse{\pedantic=1}{\marginpar{\raggedright{\textcolor{red}{#1}}}}{}}

\newtheorem{theorem}{Theorem}
\newtheorem{lemma}[theorem]{Lemma}
\newtheorem{corollary}[theorem]{Corollary}
\newtheorem{proposition}[theorem]{Proposition}

\theoremstyle{definition}
\newtheorem{definition}[theorem]{Definition}
\theoremstyle{remark}
\newtheorem{remark}[theorem]{Remark}

\newcommand{\real}{{\mathbb R}}

\newcommand{\zed}{{\mathbb Z}}
\newcommand{\euc}{{\mathbb E}}

\newcommand{\eg}{{\em e.g.}}

\newcommand{\del}{\partial}
\newcommand{\inv}{^{-1}}

\newcommand{\style}[1]{{\emph{#1}}}
\newcommand{\diff}{\Omega}  
\newcommand{\curr}{\Omega} 
\newcommand{\mass}{\mathbf{M}}   
\newcommand{\Normal}[1]{\mathbf{N}^{#1}}   
\newcommand{\Conormal}[1]{\mathbf{C}^{#1}}   
\newcommand{\rect}[1]{\mathcal{R}_{#1}}  
\newcommand{\vdual}{\mathbf{D}}    
\newcommand{\Omin}{{\mathcal O}}

\newcommand{\one}{{\mathbf{1}}}

\newcommand{\Sets}{\mathcal{S}}

\newcommand{\vx}{\mathbf{x}}
\newcommand{\CF}{\mathbf{CF}}
\newcommand{\Def}{\mathbf{Def}}

\newcommand{\dchifloor}{{\lfloor d\chi\rfloor}}
\newcommand{\dchiceil}{{\lceil d\chi\rceil}}
\newcommand{\dmufloor}[1]{{\lfloor d\mu_{#1}\rfloor}}
\newcommand{\dmuceil}[1]{{\lceil d\mu_{#1}\rceil}}

\newcommand{\Planes}[2]{{\mathcal P}_{#1,#2}}
\newcommand{\Grass}[2]{{\mathcal G}_{#1,#2}}

\newcommand{\flatnorm}[1]{\left| #1 \right|_{\flat}} 
\newcommand{\dflat}{d_{\flat}}  
\newcommand{\dflatlow}{\underline{d}_{\flat}}   
\newcommand{\dflatup}{\overline{d}_{\flat}}   

\begin{document}

\title[HADWIGER'S THEOREM FOR DEFINABLE FUNCTIONS]{HADWIGER'S THEOREM
  FOR DEFINABLE FUNCTIONS}
\thanks{This work supported by DARPA \# HR0011-07-1-0002 and by ONR N000140810668.}
\author{Y. Baryshnikov}
\address{Departments of Mathematics and Electrical and Computing Engineering,
        University of Illinois at Urbana-Champaign,
        Urbana IL, USA}
\email{ymb@uiuc.edu}

\author{R. Ghrist}
\address{Departments of Mathematics and Electrical/Systems Engineering,
        University of Pennsylvania,
        Philadelphia PA, USA}
\email{ghrist@math.upenn.edu}

\author{M. Wright}
\address{Department of Mathematics,
        Huntington University,
        Huntington IN, USA}
\email{mwright@huntington.edu}

\begin{abstract}
Hadwiger's Theorem states that $\euc_n$-invariant convex-continuous
valuations of definable sets in $\real^n$ are linear combinations of
intrinsic volumes. We lift this result from sets to data distributions
over sets, specifically, to definable $\real$-valued functions on
$\real^n$. This generalizes intrinsic volumes to (dual pairs of)
non-linear valuations on functions and provides a dual pair of
Hadwiger classification theorems.
\end{abstract}


\keywords{valuations, Hadwiger measure, intrinsic volumes, Euler characteristic.}

\maketitle


\ifthenelse{\pedantic=1}{\linenumbers}{}

\section{Introduction}

Let $\real^n$ denote Euclidean $n$-dimensional space. A
\style{valuation} on a collection $\Sets$ of subsets of $\real^n$ is
an additive function $v: \Sets \to \real$:
\begin{equation}
\label{eq:additive}
	v(A) + v(B) = v(A \cap B) + v(A \cup B) \qquad \text{whenever } A, B, A \cap B, A \cup B \in \Sets.
\end{equation}

Valuation $v$ is \style{$\euc_n$-invariant} if $v(\varphi A) = v(A)$
for all $A \in \Sets$ and $\varphi \in \euc_n$, the group of Euclidean
(or rigid) motions in $\real^n$.  A classical theorem of Hadwiger
\cite{Hadwiger} states that the $\euc_n$-invariant continuous
valuations on compact convex sets $\Sets$ in $\real^n$ (here a
valuation is \style{continuous} with respect to convergence of sets in
the Hausdorff metric) form a finite-dimensional $\real$-vector space
generated by intrinsic volumes $\mu_k$, $k=0,\ldots,n$.

\begin{theorem}[Hadwiger]
Any $\euc_n$-invariant continuous valuation $v$ on compact convex subsets of
$\real^n$ is a linear combination 
of the intrinsic volumes:
\begin{equation}
	v = \sum_{k=0}^n c_k \mu_k \, ,
\end{equation}
for some constants $c_k \in \real$. If $v$ is homogeneous of degree $k$, then $v = c_k \mu_k$.
\end{theorem}

The \style{intrinsic
  volumes}\footnote{Intrinsic volumes are also known in the literature
  as \style{Hadwiger measures}, \style{quermassintegrale},
  \style{Lipschitz-Killing curvatures}, \style{Minkowski functionals},
  and, likely, more.}  $\mu_k$ are characterized uniquely by (1)
$\euc_n$ invariance, (2) normalization with respect to a closed unit
ball, and (3) homogeneity: $\mu_k(\lambda\cdot A)=\lambda^k(A)$ for
all $A\in\Sets$ and $\lambda\in\real^+$. These measures generalize
Euclidean $n$-dimensional volume ($\mu_n$) and Euler characteristic
($\mu_0$).

This paper extends Hadwiger's Theorem to similar valuations on
functions 
instead of sets. 
Section \ref{sec:back}
gives background on the definable (o-minimal) setting that lifts
Hadwiger's Theorem to tame, non-convex sets and then to constructible
functions; there, we also review the convex-geometric,
integral-geometric, and sheaf-theoretic approaches to Hadwiger's
Theorem. In Section \ref{sec:dist}, we consider definable functions
$\Def(\real^n)$ as $\real$-valued functions with tame graphs, and
correspondingly define dual pairs of (typically) non-linear ``integral''
operators $\int\cdot\dmufloor{k}$ and $\int\cdot\dmuceil{k}$ mapping
$\Def(\real^n)\to\real$ as generalizations of intrinsic volumes, so
that $\int\one_A\dmufloor{k}=\mu_k(A)=\int\one_A\dmuceil{k}$ for all
$A$ definable. These integrals are $\euc_n$-invariant and satisfy
generalized homogeneity and additivity conditions reminiscent of
intrinsic volumes; they are furthermore compact-continuous with
respect to a dual pair of topologies on functions. This culminates
in Section \ref{sec:HT4D} in a generalization of Hadwiger's Theorem
for functions:

\begin{theorem}[Main]
Any $\euc_n$-invariant definably lower- (resp. upper-) continuous valuation $v:\Def(\real^n)\to\real$ is of the form:
\begin{equation}
	v(h) = \sum_{k=0}^n
    \left(
        \int_{\real^n} c_k\circ h \dmufloor{k}
    \right)
\end{equation}
(resp., integrals with respect to $\dmuceil{k}$) for some $c_k \in C(\real)$ continuous and monotone, satisfying $c_k(0)=0$.
\end{theorem}

The $k=0$ intrinsic measure $\dmufloor{0}$ is a recent generalization
$\dchifloor$ of Euler characteristic \cite{BG:pnas} shown to have
applications to signal processing \cite{GR} and Gaussian random fields
\cite{BoBo}. The $k=n$ intrinsic measure $\dmufloor{n}$ is Lebesgue
volume. The measures come in dual pairs $\dmufloor{k}$ and
$\dmuceil{k}$ as a manifestation of (Verdier-Poincar\'e) duality. Our
results yield the following:

\begin{corollary}
Any $\euc_n$-invariant valuation, both upper- and lower-continuous, is
a weighted Lebesgue integral.
\end{corollary}

We conclude the Introduction remarking that over the past few years several very
interesting papers by M. Ludwig, A. Tsang and others appeared, that
deal with valuations (often tensor- or set-valued) on various
functional spaces (such as $L_p$ and Sobolev spaces): for a recent report, see \cite{Ludwig}. 
Our approach deviates from this circle of results
primarily in the choice of the functional space: definable
functions form a distinctly
different domain for the valuation. The quite fine topologies we
impose on the definable functions yield a rich supply of the
valuations continuous in these topologies.

\section{Background}
\label{sec:back}

\subsection{Euler characteristic}\label{subsec:Euler}

Intrinsic volumes are built upon the Euler characteristic. Among the
many possible approaches to this topological invariant ---
combinatorial \cite{KR}, cohomological \cite{GZ}, sheaf-theoretic
\cite{Schapira,Schurmann}, we use the language of o-minimal geometry
\cite{vdD}. An \style{o-minimal structure} is a sequence
$\Omin=(O_n)_n$ of Boolean algebras of subsets of $\real^n$ which
satisfy a few basic axioms (closure under cross products and
projections; algebraic basis; and $O_1$ consists of finite unions of
points and open intervals). Examples of o-minimal structures include
the semialgebraic sets, globally subanalytic sets, and (by slight
abuse of terminology) semilinear sets; more exotic structures with
exponentials also occur \cite{vdD}. The details of o-minimal geometry
can be ignored in this paper, with the following exceptions:

\begin{enumerate}
\item Elements of $\Omin$ are called \style{tame} or, more properly,
  \style{definable} sets.
\item A mapping between definable sets is definable if and only if its
  graph is a definable set.
\item The basic equivalence relation on definable sets is definable
  bijection; these are not necessarily continuous.
\item The Triangulation Theorem \cite[Thm 8.1.7, p. 122]{vdD}: any
  definable set $Y$ is definably equivalent to a finite disjoint union
  of open simplices $\{\sigma\}$ of different dimensions.
\item The Hardt Theorem \cite{vdD}: for $f:X\to Y$ definable, $Y$ has
  a triangulation into open simplices $\{\sigma\}$ such that
  $f\inv(\sigma)$ is homeomorphic to $U_\sigma\times\sigma$ for
  $U_\sigma$ definable, and, on this inverse image, $f$ acts as
  projection.
\end{enumerate}
For more information, the reader is encouraged to consult \cite{vdD,Coste,Nicolaescu:normal}.

The \style{(o-minimal) Euler characteristic} is the valuation $\chi$
that evaluates to $(-1)^k$ on an open $k$-simplex. It is well-defined
and invariant under definable bijection \cite[Sec.~4.2]{vdD}, and,
among definable sets of fixed dimension (the dimension of the largest
cell in a triangulation), is a complete invariant of definable sets up
to definable bijection.
%
Note that the o-minimal Euler characteristic coincides with the
homological Euler characteristic (alternating sum of ranks of homology
groups) on compact definable sets. For locally compact definable sets,
it has a cohomological definition (alternating sum of ranks of
compactly-supported sheaf cohomology), yielding invariance with
respect to proper homotopy.


\subsection{Intrinsic volumes}

Intrinsic volumes have a rich history (see, \eg,
\cite{Bernig,Cheeger,KR,Schanuel}) and as many formulations as names,
including the following:

\noindent {\bf Slices:} One way to define the intrinsic volume
$\mu_k(A)$ of a definable set $A$ is in terms of the Euler characteristic of all slices of
$A$ along affine codimension-$k$ planes:
\begin{equation}
\label{eq:ivdef}
		\mu_k(A) = \int_{\Planes{n}{n-k}} \chi(A \cap P) \ d\lambda(P) ,
\end{equation}
where $\lambda$ is the following measure on $\Planes{n}{n-k}$, the
space of affine $(n-k)$-planes in $\real^n$. Each affine subspace $P
\in \Planes{n}{n-k}$ is a translation of some linear subspace $L \in
\Grass{n}{n-k}$, the Grassmannian of $(n-k)$-dimensional subspaces of
$\real^n$. That is, $P$ is uniquely determined by $L$ and a vector
$\vx \in L^\bot$, such that $P = L + \vx$. Thus, we can
integrate over $\Planes{n}{n-k}$ by first integrating over $L^\bot$
and then over
$\Grass{n}{n-k}$. 
Equation \eqref{eq:ivdef} is
equivalent to
\begin{equation}
\label{eq:ivdef2}
	\mu_k(A) = \int_{\Grass{n}{n-k}} \int_{\real^n/L} \chi(A \cap (L+\vx)) \ d\vx \ d\gamma(L) ,
\end{equation}
where $L \in\Grass{n}{n-k}$, the factorspace $\real^n/L$ is given the
natural Lebesgue measure,
and $\gamma$ is the Haar (i.e. $\mathit{SO}(n)$-invariant) measure on the Grassmannian, scaled
appropriately.

\noindent {\bf Projections:} Dual to the above definition, one can
express $\mu_k$ in terms of projections onto $k$-dimensional linear
subspaces: for any definable $A\subset\real^n$ and $0 \le k \le n$,
\begin{equation}
\label{eq:ivProjEq}
		\mu_k(A) = \int_{\Grass{n}{k}} \int_{L} \chi(\pi_L^{-1}(\vx)) \ d\vx \ d\gamma(L)
\end{equation}
where $L \in\Grass{n}{k}$ and $\pi_L^{-1}(\vx)$ is the fiber over $\vx \in L$ of the orthogonal projection map $\pi : A \to L$.
%
For $A$ convex, Equation \eqref{eq:ivProjEq} reduces to
\begin{equation*}
	\mu_k(A) = \int_{\Grass{n}{k}} \mu_k(A | L) \ d\gamma(L)
\end{equation*}
where the integrand is the $k$-dimensional (Lebesgue) volume of the
projection of $A$ onto a $k$-dimensional subspace $L$ of $\real^n$.


\subsection{Normal, conormal, and characteristic cycles}

Perspectives from geometric measure theory and sheaf theory are also
relevant to the definition of intrinsic volumes. In this section, we
restrict to the o-minimal structure of globally subanalytic sets and
use analytic tools based on geometric measure theory, following
Alesker \cite{Alesker, Alesker:gafa}, Fu \cite{Fu}, Nicolaescu
\cite{Nicolaescu:conormal, Nicolaescu:normal} and many others.

Let $\diff^k_c(\real^n)$ be the space of differential $k$-forms on
$\real^n$ with compact support. Let $\curr_k(\real^n)$ be the space of
$k$-currents --- the topological dual of $\diff_c^k(\real^n)$. Given
any $k$-current $T \in \curr_k(\real^n)$, the boundary of $T$ is $\del
T \in \curr_{k-1}(\real^n)$ defined as the adjoint to the exterior
derivative $d$. A \style{cycle} is a current with null boundary.

It is customary to use the \style{flat topology} on currents \cite{Federer}. The \style{mass} of a $k$-current $T$ is
\begin{equation}
	\mass(T) = \sup\left\{T(\omega) : \omega \in
        \curr_c^k(\real^n) \ \text{and} \sup_{x \in \real^n}
        |\omega(x)| \le 1 \right\}
\end{equation}
($|\omega|$ is the usual norm),
which generalizes the volume of a submanifold. The \style{flat norm}
of a $k$-current $T$ is
\begin{equation}\label{eq:FlatNorm}
	\flatnorm{T} = \inf\left\{ \mass(R) + \mass(S) : T = R + \del S, R \in \rect{k}(\real^n), S \in \rect{k+1}(\real^n) \right\},
\end{equation}
where $\rect{k}$ is the space of rectifiable $k$-currents. The flat
norm quantifies the minimal-mass decomposition of a $k$-current $T$
into a $k$-current $R$ and the boundary of a $(k+1)$-current $S$.

\noindent {\bf Normal cycle:} The \style{normal cycle} of a compact
definable set $A$ is a definable $(n-1)$-current $\Normal{A}$ on the
unit sphere cotangent bundle $U^*\real^n\cong S^{n-1}\times\real^n$ that is
Legendrian with respect to the canonical 1-form $\alpha$ on
$T^*\real^n$. The normal cycle generalizes the unit normal bundle of
an embedded submanifold to compact definable sets. The normal cycle is
additive: for $A$ and $B$ compact and definable,
\begin{equation}
	\Normal{A \cup B} + \Normal{A \cap B} = \Normal{A} + \Normal{B} .
\end{equation}

The intrinsic volume $\mu_k$ is representable as integration of a
particular (non-unique) form $\alpha_k\in\diff^{n-1}U^*\real^n$
against the normal cycle:
\begin{equation}\label{eq:alphas}
        \mu_k(A) = \int_{\Normal{A}}\alpha_k .
\end{equation}
Fu \cite{Fu} gives a formula for the normal cycle in terms of
stratified Morse theory; Nicolaescu \cite{Nicolaescu:normal} gives a
nice description of the normal cycle from Morse theory.

\noindent {\bf Conormal cycle:} The \style{conormal cycle} (also known
as the \style{characteristic cycle} \cite{Kashiwara,KS,Schurmann}) of
a compact definable set $A$ is a Lagrangian $n$-current $\Conormal{A}$
on $T^*\real^n$ that generalizes the cone of the unit normal
bundle. Indeed, the conormal cycle is the cone over the normal
cycle. An intrinsic description for $A$ a
submanifold-with-corners is that the conormal cycle is the union of duals to
tangent cones at points of $A$. The conormal cycle is 
additive: for $A$ and $B$ definable,
\begin{equation}
	\Conormal{A \cup B} + \Conormal{A \cap B} = \Conormal{A} + \Conormal{B} .
\end{equation}

\begin{figure}
	\begin{center}
		\includegraphics{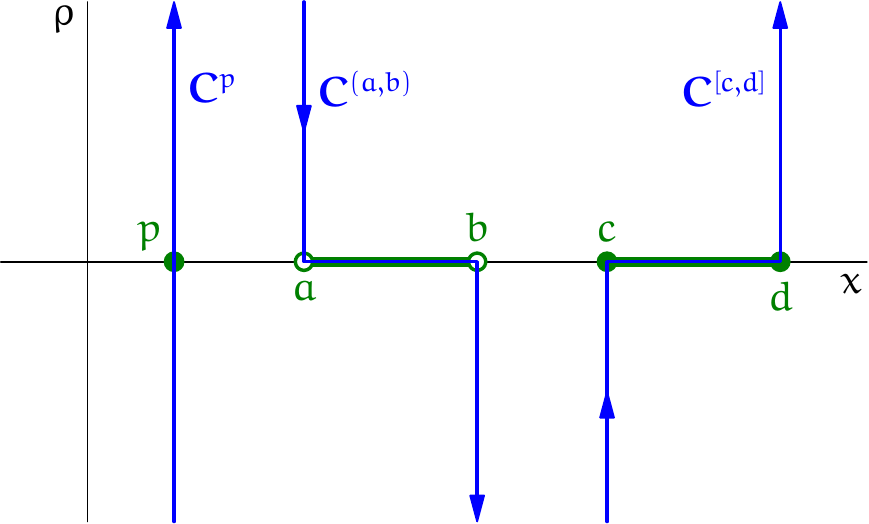}
		\caption[Conormal cycle of an interval]{Conormal
                  cycles of the point $p$, the open interval $(a,b)$,
                  and the closed interval $[c,d]$ illustrate the
                  additivity of the conormal cycle.}
		\label{ConormalCycleOfInterval}
	\end{center}
\end{figure}

The intrinsic volume $\mu_k$ is representable as integration of a
certain (non-unique) form $\omega_k\in\diff^{n}T^*\real^n$ (supported
by a bounded neighborhood of the zero section of the cotangent bundle) against the conormal cycle:
\begin{equation}\label{eq:int-conormal}
        \mu_k(A) = \int_{\Conormal{A}}\omega_k .
\end{equation}

As the conormal cycles are cones, one can always rescale the forms
$\omega_k$ so that they are supported in a given neighborhood of  the
zero section of the cotangent bundle. We fix the neighborhood once and
for all, and will assume henceforth that all $\omega_k$ are supported
in the unit ball bundle $B_1^*\real^n:=\{|P|\leq 1\}\subset T^*\real^n$.

The microlocal index theorem \cite{KS,Schurmann} gives an
interpretation of the conormal cycle in terms of stratified Morse
theory.

\noindent {\bf Continuity:} The flat norm on conormal cycles yields a
topology on definable subsets on which the intrinsic volumes are
continuous. For definable subsets $A$ and $B$, define the \style{flat
  metric} by
\begin{equation}\label{eq:FlatMetricSets}
	\dflat(A,B) = \flatnorm{(\Conormal{A} - \Conormal{B})\cap B^*_1\real^n},
\end{equation}
thereby inducing the \style{flat topology}. (That this is a metric
follows from $\Conormal{-}$ being an injection on definable subsets.)
For any $T \in \curr_n$ and $\omega \in \diff_c^n$, both supported on
$B^*_1\real^n$:

\begin{equation}\label{eq:FlatNormIneq}
	|T(\omega)| \le \flatnorm{T} \cdot \max \left\{ \sup_{B^*_1\real^n}
        |\omega| , \sup_{B^*_1\real^n} |d\omega| \right\}.
\end{equation}
Since the intrinsic volumes can be represented by integration of
bounded forms over the intersection of the conormal cycle with the
unit ball bundle, the intrinsic volumes are
continuous with respect to the flat topology. We remark also that for
the {\em convex} constructible sets, the flat topology is equivalent
to the one given by the Hausdorff metric. 

\section{Intrinsic volumes for constructible functions}
\label{sec:const}

It is possible to extend the intrinsic volumes beyond definable sets. The
\style{constructible functions}, $\CF$, are functions
$h:\real^n\to\real$ with discrete image  and definable level sets.
By abuse of terminology, $\CF$ will always refer to compactly
supported definable functions with {\em finite} image in $\real$. 

As the integral with respect to the Euler characteristics is well
defined for constructible functions, one can extend the intrinsic
volumes to constructible functions using the slicing definition
above:
\begin{equation}
\label{eq:IVforCF}
		\mu_k(h) = \int_{\Grass{n}{n-k}} \int_{\real^n/L} \left(\int_{L+\vx}
                hd\chi\right) \ d\vx \ d\gamma(L).
\end{equation}

In so doing, one obtains, \eg, the following generalization of the Poincar\'e theorem for Euler characteristic.

We need the Verdier duality operator in $\CF$, which is defined e.g. in \cite{Schapira}. Briefly, the dual of $h \in \CF$ is a function $\vdual h$ whose value at $x_0$ is given by
\begin{equation}
	(\vdual h)(x_0) = \lim_{\epsilon \to 0} \int_{\real^n} \one_{B(x_0,\epsilon)} h \ d\chi,
\end{equation}
where the integral is with respect to Euler characteristic (see also
\cite{BG:siam}), and $B(x_0,\epsilon)$ is the $n$-dimensional ball of
radius $\epsilon$ centered at $x_0$. In many cases, this duality
swaps interiors and closures. For example, if $A$ is a convex open
set with closure $\overline{A}$, then $\vdual \one_A =
\one_{\overline{A}}$ and $\vdual \one_{\overline{A}} = \one_A$.

\begin{proposition}\label{prop:generalPoincare}
	For a constructible function $h$ on $\real^n$, $h\in\CF$, and $\vdual$ the Verdier duality operator in $\CF$,
	\begin{equation}
		\int_{\real^n} h \ d\mu_k = (-1)^{n-k} \int_{\real^n}
                \vdual h \ d\mu_k.
	\end{equation}
\end{proposition}
\begin{proof}
The result holds in the case $k=0$ (see \cite{Schapira}). 
From Equation (\ref{eq:IVforCF}), $\mu_k$ is defined by integration
with respect to $d\chi$ along codimension-$k$ planes, followed by the
integration over the planes. By Sard's theorem, for (Lebesgue) almost all
$L\in\Grass{n}{n-k}$ and $\vx\in {\real^n/L}$, the level sets of $h$
are transversal to $L+\vx$, whence, by Thom's second isotopy lemma,
\cite{shiota}, 
 	\begin{equation}
		\int_{L+\vx} h \ d\chi = (-1)^{n-k} \int_{L+\vx}
                \vdual h \ d\chi
	\end{equation}
for almost all $L$ and $\vx$. Integration over $\Planes{n}{n-k}$ finishes the proof.
\end{proof}

\begin{remark}
	If definable sets $A$, $\overline{A} \subset \real^n$ satisfy $\vdual \one_A = \one_{\overline{A}}$, then Proposition \ref{prop:generalPoincare} implies
	\begin{equation}\label{eq:PoincareForSets}
		\mu_k(A) = (-1)^{n-k} \mu_k(\overline{A}).
	\end{equation}
\end{remark}



\section{Intrinsic volumes for definable functions}
\label{sec:dist}

The next logical step, lifting from constructible to definable
functions, is the focus of this paper. Let $\Def(\real^n)$ denote the
\style{definable functions} on $\real^n$, that is, the set of
functions $h : \real^n \to \real$ whose graphs
are definable sets in $\real^n\times\real$ which coincide with
$\real^n\times\{0\}$ outside of a ball (thus compactly supported and
bounded). 
In \cite{BG:pnas},
integration with respect to Euler characteristic $\mu_0$ was lifted to
a dual pair of nonlinear ``integrals'' $\int\cdot\dchifloor$ and
$\int\cdot\dchiceil$ via the following limiting process, now extended
to $\mu_k$:

\begin{definition}
\label{def:int}
For $h \in \Def(\real^n)$, the \style{lower} and \style{upper Hadwiger integrals} of $h$ are, respectively,
	\begin{align}
		\int h\ \dmufloor{k} &= \lim_{m \to \infty} \frac{1}{m} \int \lfloor mh \rfloor \ d\mu_k, \text{ and} \\
		\int h\ \dmuceil{k} &= \lim_{m \to \infty} \frac{1}{m} \int \lceil mh \rceil \ d\mu_k . \notag
	\end{align}
\end{definition}

For $k=n$ these two definitions agree with each other and with the
Lebesgue integral; for all $k<n$, they differ. For $k=0$, these become
the definable Euler integrals $\int\cdot\dchifloor$ and
$\int\cdot\dchiceil$. The following result demonstrates several
equivalent formulations, mirroring those of Section \ref{sec:back}. As
a consequence, the limits in Definition \ref{def:int} are
well-defined, following from compact support and the well-definedness
of $\dchifloor$ and $\dchiceil$ from \cite{BG:pnas}.

\begin{theorem}\label{thm:integrals}
	For $h \in \Def(\real^n)$,
	\begin{align}
		\int h\ \dmufloor{k}
		\label{eq:ex}        &= \int_{s=0}^\infty \mu_k\{ h \ge s \} - \mu_k \{ h < -s \} \ ds & &\text{excursion sets} \\
		\label{eq:slice}        &= \int_{\Planes{n}{n-k}} \int_{P} h \ \dchifloor \ d\lambda(P) & &\text{slices} \\
		\label{eq:proj}        &= \int_{G_{n,k}} \int_L \int_{\pi_L^{-1}(\vx)} h \ \dchifloor d\vx \ d\gamma(L) & &\text{projections} \\
		\label{eq:co}        &= \int_{s=0}^\infty \left( \Conormal{\{ h \ge s\}}(\omega_k) - \Conormal{\{ h < -s\}}(\omega_k) \right) \ ds  & &\text{conormal cycle} \\
		\label{eq:dual}        &=-\int -h\dmuceil{k} & &\text{duality}
	\end{align}	
\end{theorem}
\begin{proof}
Note that for $T>0$ sufficiently large and $N=mT$,
	\begin{align*}
		\int h\ \dmufloor{k} &= \lim_{m \to \infty} \frac{1}{m} \int \lfloor mh \rfloor \ d\mu_k = \lim_{m \to \infty} \frac{1}{m} \sum_{i=1}^\infty \mu_k \{ mh \ge i \} - \mu_k \{ mh < -i \} \\
		&= \lim_{N \to \infty} \frac{T}{N} \sum_{i=1}^N \mu_k \left\{ h \ge \frac{iT}{N} \right\} - \mu_k \left\{ h < -\frac{iT}{N} \right\} \\
		&= \int_0^T \mu_k\{ h \ge s \} - \mu_k\{ h < -s \} \ ds.
	\end{align*}
Thus, (\ref{eq:ex}); the same proof using $\dmuceil{k}$ implies that
	\begin{equation}
		\int h\ \dmuceil{k} = \int_{s=0}^\infty \mu_k\{ h > s \} - \mu_k \{ h \le -s \} \ ds,
	\end{equation}
which, with (\ref{eq:ex}), yields (\ref{eq:dual}). For (\ref{eq:slice}),
	\begin{equation*}
		\int_0^\infty \mu_k \{ h \ge s \} - \mu_k \{ h < -s \} \ ds = \int_0^\infty \int_{\Planes{n}{n-k}} \chi (\{ h \ge s \} \cap P) - \chi (\{ h < -s \} \cap P) \ d\lambda(P) \ ds.
	\end{equation*}
This integral is well-defined, since the excursion sets $\{ h \ge s\}$ and $\{ h < -s \}$ are definable, and $h$ is bounded and of compact support. The Fubini theorem yields (\ref{eq:slice}) via
	\begin{equation*}
		\int_{\Planes{n}{n-k}} \int_0^\infty \chi (\{ h \ge s \} \cap P) - \chi (\{ h < -s \} \cap P) \ ds \ d\lambda(P) = \int_{\Planes{n}{n-k}} \int_{P} h \ \dchifloor \ d\lambda(P).
	\end{equation*}
For (\ref{eq:proj}), fix an $L \in G_{n,k}$ and let $\pi_L$ be the orthogonal projection map on to $L$. Then the affine subspaces perpendicular to $L$ are the fibers of $\pi_L$ and
	\begin{equation*}
		\{ P \in \Planes{n}{n-k} : P \bot L \} = \{ \pi_L^{-1}(\vx) : \vx \in L \}.
	\end{equation*}
Instead of integrating over $\Planes{n}{n-k}$, integrate over the fibers of orthogonal projections onto all linear subspaces of $G_{n,k}$:
	\begin{equation*}
		\int h \ \dmufloor{k} = \int_{\Planes{n}{n-k}} \int_{P} h \ \dchifloor \ d\lambda(P) = \int_{G_{n,k}} \int_L \int_{\pi_L^{-1}(\vx)} h \ \dchifloor \ d\vx \ d\gamma(L).
	\end{equation*}
Finally, for (\ref{eq:co}), rewrite (\ref{eq:ex}) by expressing the intrinsic volumes in terms of the conormal cycles, as in (\ref{eq:int-conormal}).
\end{proof}

\begin{figure}
	\begin{center}
		\includegraphics{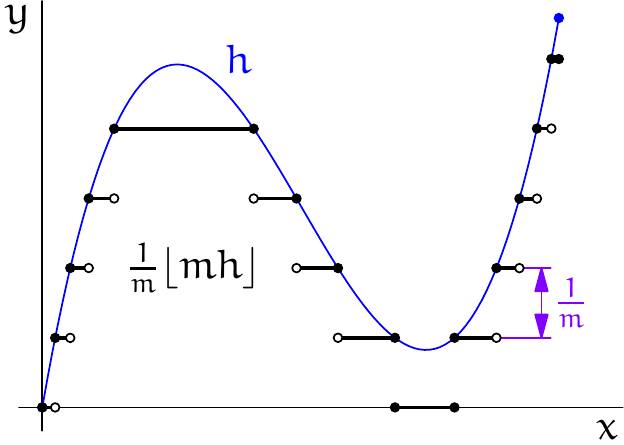} \hspace{.5in} \includegraphics{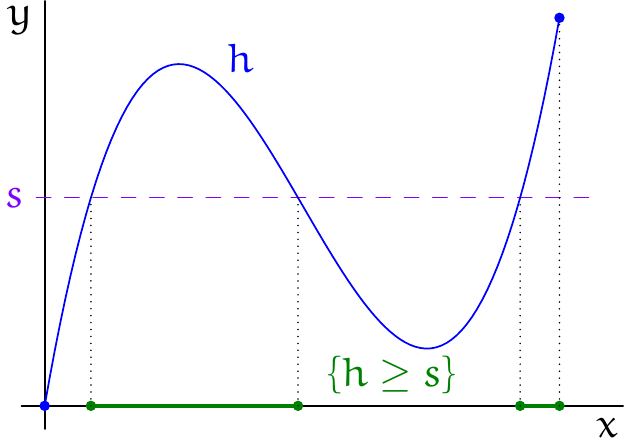}
		\caption[Hadwiger integral definition]{The lower Hadwiger integral is defined as a limit of lower step functions (left), as in Definition \ref{def:int}. It can also be expressed in terms of excursion sets (right), as in Theorem \ref{thm:integrals}, equation \eqref{eq:ex}.}
		\label{fig:HadIntDef}
	\end{center}
\end{figure}

\section{Continuous valuations}
\label{sec:cont}

Valuations on functions are a straightforward generalization of
valuations on sets. A \style{valuation} on $\Def(\real^n)$ is a
functional $v : \Def(\real^n) \to \real$, satisfying $v(0) = 0$ and the
following \style{additivity} condition:
\begin{equation}\label{eq:addfunc}
    v(f) + v(g) = v(f \vee g) + v(f \wedge g) ,
\end{equation}
where $\vee$ and $\wedge$ denote the pointwise max and min,
respectively.

We present two useful topologies on $\Def(\real^n)$ that allow us to
consider \style{continuous} valuations. With these topologies, the
notion of a continuous valuation on $\Def(\real^n)$ properly extends
the notion of a continuous valuation on definable subsets of
$\real^n$.

\begin{definition}\label{def:FlatTopoFunc}
Let $f, g, \in \Def(\real^n)$. The \style{lower} and \style{upper flat
  metrics} on definable functions, denoted $\dflatlow$ and $\dflatup$,
respectively, are defined as follows (see \ref{eq:FlatMetricSets}):
	\begin{align}
		\dflatlow(f,g) &= \int_{-\infty}^\infty
                \dflat(\Conormal{\{f \ge s\}},\Conormal{\{g \ge s\}})
                \ ds \quad \text{and} \label{eq:LowFlatMetricFunct}
                \\ \dflatup(f,g) &= \int_{-\infty}^\infty
                \dflat(\Conormal{\{f > s\}},\Conormal{\{g > s\}})
                \ ds. \label{eq:UprFlatMetricFunct}
	\end{align}
The distinct topologies induced by the lower and upper flat metrics
are the \style{lower} and \style{upper flat topologies} on definable
functions.  A valuation on definable functions is \style{lower-} or
\style{upper-continuous} if it is continuous in the lower or upper
flat topology, respectively.
\end{definition}

Note that the integrals in \eqref{eq:LowFlatMetricFunct} and
\eqref{eq:UprFlatMetricFunct} are well-defined because they may be
written with finite bounds, as it suffices to integrate between the
minimum and maximum values of $f$ and $g$. These metrics extend the
flat metric on definable sets, for they reduce to
\eqref{eq:FlatMetricSets} when $f$ and $g$ are characteristic
functions.

\begin{remark}\label{rem:FlatTopoMetric}
Definition \ref{def:FlatTopoFunc} does result in \style{metrics}. If
$\dflatlow(f,g)=0$, then $\flatnorm{ \Conormal{\{f \ge s\}} -
  \Conormal{\{g \ge s\}} } = 0$ only for $s$ in a set of Lebesgue
measure zero.  However, if the excursion sets of $f$ and $g$ agree
almost everywhere, then \style{all} excursion sets of $f$ and $g$
agree, and thus $f=g$.  For, if $\{s_i\}_i$ is a sequence of negative
real numbers converging $0$, and $\{f \ge s_i\} = \{g \ge s_i\}$ for
all $i$, then:
\begin{equation*}
	\{f \ge 0\} = \bigcap_i \{f \ge s_i\} = \bigcap_i \{g \ge s_i\} = \{g \ge 0\}.
\end{equation*}
The result for $\dflatup(f,g)$ follows similarly from the observation that $\{f > 0\} = \bigcup_{s>0} \{f > s\}$.
\end{remark}

\begin{remark}\label{rem:FlatTopoDistinct}
That the lower and upper flat topologies are distinct can be seen by noting that for the identity function $f$ on the interval $[0,1]$, the sequence of lower step functions $g_m = \frac{1}{m}\left\lfloor mf \right\rfloor$ converges (as $m \to \infty$) to $f$ in the lower flat topology, but not in the upper flat topology.
Dually, upper step functions converge in the upper flat topology, but not in the lower.
\end{remark}

\begin{lemma}\label{lem:HadIntCont}
The lower and upper Hadwiger integrals are lower- and upper-continuous, respectively.
\end{lemma}
\begin{proof}
Let $f, g \in \Def(\real^n)$ be supported on $X \subset \real^n$. The following inequality for the lower integrals is via \eqref{eq:FlatNormIneq}:
	\begin{align*}
		\left| \int f \dmufloor{k} - \int g \dmufloor{k} \right| &= \left| \int_{-\infty}^{\infty} (\mu_k\{ f \ge s \} - \mu_k\{ g \ge s \}) \ ds \right| \\
		&\le \int_{-\infty}^\infty \flatnorm{\left(\Conormal{\{ f
                    \ge s \}} - \Conormal{\{ g \ge s \}}\right)\cap B^*_1\real^n} \cdot 
\max \left\{ \sup_{B^*_1\real^n}
        |\omega| , \sup_{B^*_1\real^n} |d\omega| \right\}\\
		&= \dflatlow(f,g) \cdot 
\max \left\{ \sup_{B^*_1\real^n}
        |\omega| , \sup_{B^*_1\real^n} |d\omega| \right\}
	\end{align*}
Since $\omega_k$ and $d\omega_k$ are bounded, 
we have continuity of the lower integrals in the lower flat topology. The proof for the upper integrals is analogous.
\end{proof}

For \style{constructible} functions, the lower and upper flat
topologies of the previous section are equivalent.  Thus, we may refer
to the \style{flat topology} on constructible functions without
specifying \style{upper} or \style{lower}.  A valuation on
constructible functions is \style{conormal continuous} if it is
continuous with respect to the flat topology.  Conormal continuity is
the same as ``smooth'' in the Alesker sense \cite{Alesker,
  Alesker:gafa}, but distinct from continuity in the topology induced
by the Hausdorff metric on definable sets.

\section{Hadwiger's Theorem for functions}
\label{sec:HT4D}

A dual pair of Hadwiger-type classifications for (lower-/upper-)
continuous Euclidean-invariant valuations is the goal of this paper.

\begin{lemma}\label{lem:HadwigerConstrFunc}
If $v : \CF(\real^n) \to \real$ is a (conormal) continuous valuation on
constructible functions, invariant with respect to the right action by
Euclidean motions, 
then $v$ is of the form:
	\begin{equation*}
		v(h) = \sum_{k=0}^n \int_{\real^n} c_k(h) \ d\mu_k.
	\end{equation*}
for some coefficient functions $c_k : \real \to \real$ with $c_k(0)=0$.
\end{lemma}
\begin{proof}
For the class of indicator functions for {\em convex sets}
$\{h = r \cdot \one_A : r \in
\zed \text{ and } A \subset \real^n \text{ definable} \}$, continuity
of $v$ in the flat topology implies that $v$ is
continuous in the Hausdorff topology.  Since 
convex tame sets are dense (in Hausdorff metric) among convex sets in $\real^n$,
Hadwiger's Theorem for sets implies that
	\begin{equation}\label{eq:HadwigerConstrFunc1}
		v(r \cdot \one_A) = \sum_{k=0}^n c_k(r)\mu_k(A),
	\end{equation}
where $c_k(r)$ are constants that depend only on $v$, not on
$A$. Conormal continuity implies that the valuation $v(A)$ is the
integral of the linear combination of the forms $\omega_k$ (defined in \eqref{eq:int-conormal}),
		\begin{equation}\label{eq:HadwigerConstrFuncForms}
		\sum_{k=0}^n c_k(r)\alpha_k
	\end{equation}	
over $\Conormal{A}$.

Now suppose $h = \sum_{i=1}^m r_i \one_{A_i}$ is a finite sum of
indicator functions of 
disjoint definable subsets $A_1, \ldots, A_m$ of $\real^n$ for some integer constants $r_1 < r_2 < \cdots < r_m$. By equation \eqref{eq:HadwigerConstrFunc1} and additivity,
	\begin{equation}\label{eq:HadwigerConstrFunc2}
		v(h) = \sum_{k=0}^n \sum_{i=1}^m c_k(r_i) \mu_k(A_i).
	\end{equation}
	
We can rewrite equation \eqref{eq:HadwigerConstrFunc2} in terms of excursion sets of $h$. Let $B_i = \bigcup_{j \ge i} A_j$. That is, $B_i = \{ h \ge r_i \}$ and $B_i = \{ h > r_{i-1} \}$. Then the valuation $v(h)$ can be expressed as:
	\begin{equation}\label{eq:HadwigerConstrFunc3}
		v(h) = \sum_{k=0}^n \sum_{i=1}^m \left( c_k(r_i) - c_k(r_{i-1}) \right) \mu_k(B_i),
	\end{equation}
where $c_k(r_0)=0$. Thus, a valuation of a constructible function can be expressed as a sum of finite differences of valuations of its excursion sets. Equivalently, equation \eqref{eq:HadwigerConstrFunc3} can be written in terms of constructible Hadwiger integrals:
	\begin{equation}\label{eq:ValuationFunctional3}
		v(h) = \sum_{k=0}^n \int_{\real^n} c_k(h) \ d\mu_k.
	\end{equation}
Since we require that a valuation of the zero function is zero, it must be that $c_k(0)=0$ for all $k$.
\end{proof}

Note that Lemma \ref{lem:HadwigerConstrFunc} holds for functions of the form $h = \sum_{i=1}^m r_i \one_{A_i}$ where the $A_i$ are definable and the $r_i \in \real$ are not necessarily integers.

In writing an arbitrary valuation on definable functions as a sum of Hadwiger integrals, the situation becomes complicated if the coefficient functions $c_k$ are decreasing on any interval. The following proposition illustrates the difficulty:

\begin{proposition}\label{prop:DecreasingComposition}
Let $c : \real \to \real$ be a continuous, strictly decreasing function. Then,
	\begin{equation}\label{eq:DecreasingComposition}
		\lim_{m \to \infty} \int_{\real^n} c\left( \frac{1}{m} \lceil mh \rceil \right) \ d\mu_k = \lim_{m \to \infty} \int_{\real^n} \frac{1}{m} \lfloor m c(h) \rfloor \ d\mu_k.
	\end{equation}
\end{proposition}
\begin{proof}
On the left side of equation \eqref{eq:DecreasingComposition}, we integrate $c$ composed with upper step functions of $h$:
	\begin{equation*}
		\int_{\real^n} c\left( \frac{1}{m} \lceil mh \rceil \right) \ d\mu_k = \sum_{i \in \zed} c\left(\tfrac{i}{m}\right) \cdot \mu_k\left\{ \tfrac{i-1}{m} < h \le \tfrac{i}{m} \right\}.
	\end{equation*}
On the right side of equation \eqref{eq:DecreasingComposition}, we integrate lower step functions of the composition $c(h)$:
	\begin{equation*}
		\int_{\real^n} \frac{1}{m} \lfloor m c(h) \rfloor \ d\mu_k = \sum_{t \in \zed} \tfrac{t}{m} \cdot \mu_k\left\{ \tfrac{t}{m} \le c(h) < \tfrac{t+1}{m} \right\}.
	\end{equation*}
Since $c$ is strictly decreasing, $c^{-1}$ exists. There exists a discrete set
	\[ \mathcal{S} = \left\{ c^{-1}\left( \tfrac{t}{m} \right) \ \big| \ t \in \zed \right\} \cap \{ \text{neighborhood around range of $h$} \}.\]
We may then rewrite the above sum as:
	\begin{equation*}
		\int_{\real^n} \frac{1}{m} \lfloor m c(h) \rfloor \ d\mu_k = \sum_{s \in \mathcal{S}} c(s) \cdot \mu_k \{ c(s) \le c(h) < c(s - \epsilon) \}
		= \sum_{s \in \mathcal{S}} c(s) \cdot \mu_k \{ s - \epsilon < h \le s \},
	\end{equation*}
where $\epsilon \to 0$ as $m \to \infty$ by continuity of $c$. In the limit, both sides are equal:
	\begin{equation*}
		\lim_{\epsilon \to 0} \sum_{s \in \mathcal{S}} c(s) \cdot \mu_k \{ s - \epsilon < h \le s \} \} = \lim_{m \to \infty} \sum_{i \in \zed} c\left(\tfrac{i}{m}\right) \cdot \mu_k\left\{ \tfrac{i-1}{m} < h \le \tfrac{i}{m} \right\}. \qedhere
	\end{equation*}
\end{proof}
\begin{figure}
	\begin{center}
		\includegraphics{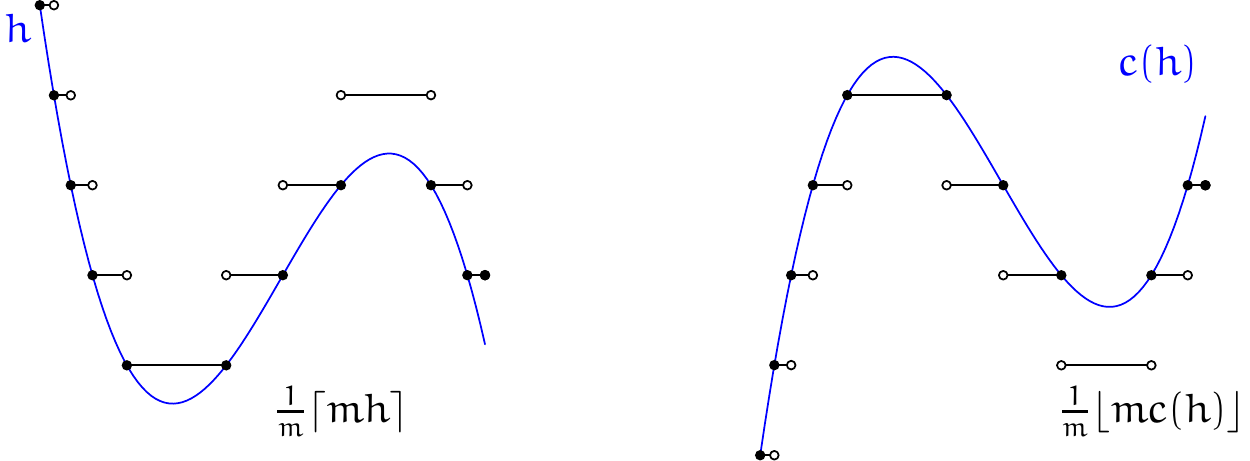}
		\caption[Decreasing composition]{An upper step function of $h$, depicted at left, composed with a decreasing function $c$, becomes a lower step function of $c(h)$, depicted at right. As the step size approaches zero, we obtain Proposition \ref{prop:DecreasingComposition}.}
		\label{fig:StepFunctions}
	\end{center}
\end{figure}

Proposition \ref{prop:DecreasingComposition} implies that if $c : \real \to \real$ is increasing on some interval and decreasing on another, then the maps $v, u : \Def(\real^n) \to \real$ defined
\begin{equation*}
	v(h) = \int_{\real^n} c(h) \dmufloor{k} \quad \text{and} \quad u(h) = \int_{\real^n} c(h) \dmuceil{k}
\end{equation*}
are neither lower- nor upper-continuous.

Lemma \ref{lem:HadwigerConstrFunc} and Proposition \ref{prop:DecreasingComposition} provide a generalization of Hadwiger's Theorem:

\begin{theorem}\label{thm:main}
Any $\euc_n$-invariant definably lower-continuous valuation $v:\Def(\real^n)\to\real$ is of the form:
	\begin{equation}
		v(h) = \sum_{k=0}^n \left( \int_{\real^n} c_k\circ h \dmufloor{k} \right) \, ,
	\end{equation}
for some $c_k \in C(\real)$ continuous and monotone, satisfying $c_k(0)=0$. Likewise, an upper-continuous valuation can be similarly written in terms of upper Hadwiger integrals.
\end{theorem}
\begin{proof}
Let $v : \Def(\real^n) \to \real$ be a lower valuation, and $h \in \Def(\real^n)$. First approximate $h$ by lower step functions. That is, for $m > 0$, let $h_m = \frac{1}{m} \lfloor mh \rfloor$. In the lower flat topology, $\lim_{m \to \infty} h_m = h$. On each of these step functions, Lemma \ref{lem:HadwigerConstrFunc} implies that $v$ is a linear combination of Hadwiger integrals:
	\begin{equation}\label{eq:HDF1}
		v(h_m) = \sum_{k=0}^n \int_{\real^n} c_k(h_m) \ d\mu_k.
	\end{equation}
for some $c_k : \real \to \real$ with $c_k(0)=0$, depending only on $v$ and not on $m$. By Proposition \ref{prop:DecreasingComposition}, the $c_k$ must be increasing functions since we are approximating $h$ with lower step functions in the lower flat topology.
	
We can alternately express equation \eqref{eq:HDF1} as
	\begin{equation}\label{eq:HDF2}
		v(h_m) = \sum_{k=0}^n \int_{\real^n} c_k(h_m) \ \dmufloor{k},
	\end{equation}
where we choose lower rather than upper integrals since $v$ is continuous in the lower flat topology. Continuity of $v$, and convergence of $h_m$ to $h$, in the lower flat topology imply that $v(h_m)$ converges to $v(h)$ as $h \to \infty$. More specifically,
	\begin{equation}\label{eq:HDF3}
		v(h) = \lim_{m \to \infty} v\left(  h_m \right) = \sum_{k=0}^n \lim_{m \to \infty} \int_{\real^n} c_k\left( h_m \right) \ \dmufloor{k}.
	\end{equation}
By continuity of the lower Hadwiger integrals (Lemma \ref{lem:HadIntCont}) and the $c_k$, Equation \eqref{eq:HDF3} becomes
	\begin{equation}\label{eq:HDF4}
		v(h) = \sum_{k=0}^n \int_{\real^n} c_k\left(\lim_{m \to \infty} h_m \right) \ \dmufloor{k} = \sum_{k=0}^n \int_{\real^n} c_k(h) \ \dmufloor{k}.
	\end{equation}
The proof for the upper valuation is analogous.
\end{proof}

\begin{corollary}
Any $\euc_n$-invariant valuation both upper- and lower-continuous is a weighted Lebesgue integral.
\end{corollary}
\begin{proof}
Integration with respect to $\dmufloor{k}$ and $\dmuceil{j}$ are independent unless $k=j=n$. 
For any $v$ both upper- and lower-continuous, we have
	\begin{equation*}
		v(h) = \sum_{k=0}^n \int_{\real^n} \underline{c}_k(h) \ \dmufloor{k} = \sum_{k=0}^n \int_{\real^n} \overline{c}_k(h) \ \dmuceil{k}
	\end{equation*}
for some functions $\underline{c}_k$ and $\overline{c}_k$.
	
Lower and upper Hadwiger integrals with respect to $\mu_k$ are unequal, except when $k=n$, implying that $\underline{c}_k = \overline{c}_k = 0$ for $k = 0, 1, \ldots, n-1$, and $\underline{c}_n = \overline{c}_n$. Therefore,
	\begin{equation*}
		v(h) = \int_{\real^n} c(h) \ d\mu_n
	\end{equation*}
for some continuous function $c : \real \to \real$, and with $d\mu_n = \dmufloor{n} = \dmuceil{n}$ denoting Lebesgue measure.
\end{proof}

\section{Speculation}
\label{sec:rem}

The present constructions are potentially applicable to generalizations of current applications of intrinsic volumes. One such recent application is to the dynamics of cellular structures, such as crystals and foams in microstructure of materials. The cells in such structures often change shape and size over time in order to minimize the total energy level in the system. Let $C = \bigcup_{i=0}^n C_i$ be a closed $n$-dimensional cell, with $C_i$ denoting the union of all $i$-dimensional features of the cell: {\em i.e.}, $C_0$ is the set of vertices, $C_1$ the set of edges, etc. MacPherson and Srolovitz found that when the cell structure changes by a process of \style{mean curvature flow}, the volume of the cell changes according to
\begin{equation}\label{eq:MPS-cell}
	\frac{d\mu_n}{dt}(C) = -2\pi M\gamma \left( \mu_{n-2}(C_{n}) - \frac{1}{6}\mu_{n-2}(C_{n-2}) \right)
\end{equation}
where $M$ and $\gamma$ are constants determined by the material properties of the cell structure \cite{MS}. Replacing the intrinsic volumes of cells with Hadwiger integrals may (1) lead to interesting dynamical systems on the (singular) foliations (by the level sets of a piece-wise smooth function, and (2) allow for description of evolution of real-valued physical fields (temperature, density, etc.) of cells.

A more widely-known application of the intrinsic volumes is in the formulas for expected Euler characteristic of excursion sets in Gaussian random fields \cite{Adler,AdlerTaylor}. These formulae and the associated Gaussian kinematic formula \cite{AdlerTaylor} rely crucially on the intrinsic volumes of excursion sets. It is already recognized in recent work \cite{BoBo} that the definable Euler measure $\dchifloor=\dmufloor{0}$ is relevant to Gaussian random fields: we strongly suspect that the other definable Hadwiger measures $\dmufloor{k}$ and $\dmuceil{k}$ of this paper are immediately applicable to Gaussian random fields.


\bibliographystyle{amsalpha}

\begin{thebibliography}{99}

\bibitem{Adler} R. Adler, {\em The Geometry of Random Fields}, Wiley, 1981; reprinted by SIAM, 2009.

\bibitem{AdlerTaylor} R. Adler and J. Taylor, ``Topological Complexity of Random Functions'', {\em Springer Lecture Notes in Mathematics}, Vol. 2019, Springer, 2011.

\bibitem{Alesker} S. Alesker, ``Theory of valuations on manifolds: a survey,'' {\em Geometric and Functional Analysis}, 17(4), 2007, 1321--1341.

\bibitem{Alesker:gafa} S. Alesker, ``Valuations on manifolds and integral geometry,'' {\em Geometric and Functional Analysis}, 20(5), 2010, 1073--1143.

\bibitem{Bernig} A. Bernig, ``Algebraic Integral Geometry,'' {\em Global Differential Geometry}, edited by C B\"{a}r, J. Lohkamp, and M. Schwarz, Springer, 2012.

\bibitem{BG:siam} Y. Baryshnikov and R. Ghrist, ``Target enumeration via Euler characteristic integration,'' {\em SIAM J. Appl. Math.},  70(3), 2009, 825--844.

\bibitem{BG:pnas} Y. Baryshnikov and R. Ghrist, ``Definable Euler integration,'' {\em Proc. Nat. Acad. Sci.},  107(21), May 25, 9525-9530, 2010.


\bibitem{BoBo} O. Bobrowski and M. Strom Borman, ``Euler Integration of Gaussian Random Fields and Persistent Homology,'' 2011, {\tt arXiv:1003.5175}.



\bibitem{Cheeger} J. Cheeger, W. M\"uller, and R. Schrader, ``On the curvature of piecewise flat spaces,'' {\em Comm. Math. Phys.} 92(3), 1984, 405--454.

\bibitem{Coste} M. Coste, An Introduction to o-minimal Geometry, {\em Dip. Mat. Univ. Pisa, Dottorato di Ricerca in Matematica, Istituti Editoriali e Poligrafici Internazionali, Pisa}, 2000, {\tt http://www.ihp-raag.org/publications.php}.

\bibitem{Federer} H. Federer, {\em Geometric Measure Theory}, Springer 1969.

\bibitem{Fu} J. Fu, ``Curvature measures of subanalytic sets'', {\em Amer. J. Math.}, 116, (1994), 819-890.

\bibitem{Fu:notes} J. Fu, ``Notes on Integral Geometry,'' 2011, {\tt http://www.math.uga.edu/\textasciitilde{}fu/notes.pdf}.

\bibitem{GR} R. Ghrist and M. Robinson, ``Euler-Bessel and Euler-Fourier transforms,'' {\em Inv. Prob.}, to appear.



\bibitem{GZ} Guesin-Zade, ``Integration with respect to the Euler characteristic and its applications,'' {\em Russ. Math. Surv.}, 65:3, 2010, 399--432.

\bibitem{Hadwiger} H. Hadwiger, ``Integrals\"atze im Konvexring,'' {\em Abh. Math. Sem. Hamburg}, 20, 1956, 136--154.

\bibitem{KR} D. A. Klain and G.-C. Rota, {\em Introduction to Geometric Probability}, Cambridge, 1997.

\bibitem{Kashiwara} M. Kashiwara, ``Index theorem for constructible sheaves,'' {\em Ast�risque}, 130, 1985, 193--209.

\bibitem{Klain} D. A. Klain, ``A Short Proof of Hadwiger's Characterization Theorem,'' {\em Mathematika}, 42, 1995, 329--339.

\bibitem{KS} M. Kashiwara and P. Schapira, {\em Sheaves on Manifolds},
  Springer, 1990.

\bibitem{Ludwig} {M. Ludwig},
     ``Valuations on function spaces,''
   {\em Adv. Geom.},
    {11},
      (2011), 745--756.


\bibitem{MS} R. D. MacPherson and D. J. Srolovitz, ``The von Neumann relation generalized to coarsening of three-dimensional microstructures,'' {\em Nature}, 446, 2007, 1053--1055.


\bibitem{Nicolaescu:conormal} L. I. Nicolaescu, ``Conormal Cycles of Tame Sets,'' preprint, 2010, {\tt http://www.nd.edu/\textasciitilde{}lnicolae/conormal.pdf}.

\bibitem{Nicolaescu:normal} L. I. Nicolaescu, ``On the Normal Cycles of Subanalytic Sets,'' {\em Ann. Glob. Anal. Geom.} 39, 2011, 427--454.



\bibitem{Schapira} P. Schapira, ``Operations on constructible functions,'' {\em J. Pure Appl. Algebra}, 72, 1991, 83--93.



\bibitem{Schurmann} J. Sch\"urmann, {\em Topology of Singular Spaces and Constructible Sheaves}, Birkh\"auser, 2003.

\bibitem{Schanuel} S. H. Schanuel, ``What is the Length of a Potato?'' in {\em Lecture Notes in Mathematics}, Springer, 1986, 118--126.

\bibitem{shiota} {M. Shiota}, {\em Geometry of subanalytic and
    semialgebraic sets}, Birkh\"auser, 1997.


\bibitem{vdD} L. Van den Dries, {\em Tame Topology and O-Minimal Structures}, Cambridge University Press, 1998.


\end{thebibliography}

\end{document}